\documentclass[article,12pt]{amsart}

\usepackage{amsfonts}
\usepackage{amsmath}
\usepackage{graphicx}
\usepackage{color}
\usepackage{epsfig}

\numberwithin{equation}{section}
\theoremstyle{plain}
\newtheorem{thm}{Theorem}[section]
\newtheorem{rem}{Remark}[section]

\newtheorem{cor}{Corollary}[section]
\newtheorem{lem}{Lemma}[section]

\newcommand{\dd}{\: \mathrm{d}}
\newcommand{\dE}{\mathbb{E}}
\newcommand{\dR}{\mathbb{R}}

\newcommand{\dP}{\mathbb{P}}
\newcommand{\dC}{\mathbb{C}}
\newcommand{\dN}{\mathbb{N}}

\newcommand{\cD}{\mathcal{D}}
\newcommand{\cJ}{\mathcal{J}}
\newcommand{\cL}{\mathcal{L}}
\newcommand{\cN}{\mathcal{N}}
\newcommand{\cK}{\mathcal{K}}
\newcommand{\rI}{\mathrm{I}}
\newcommand{\cF}{\mathcal{F}}

\newcommand{\cH}{\mathcal{H}}
\newcommand{\cR}{\mathcal{R}}
\newcommand{\cZ}{\mathcal{Z}}

\newcommand{\veps}{\varepsilon}
\newcommand{\wh}{\widehat}
\newcommand{\vp}{\varphi}
\newcommand{\wt}{\widetilde}
\newcommand{\ind}{\mbox{1}\kern-.25em \mbox{I}}
\font\calcal=cmsy10 scaled\magstep1
\def\build#1_#2^#3{\mathrel{\mathop{\kern 0pt#1}\limits_{#2}^{#3}}}
\def\liml{\build{\longrightarrow}_{}^{{\mbox{\calcal L}}}}

\def\videbox{\mathbin{\vbox{\hrule\hbox{\vrule height1.4ex \kern.6em\vrule height1.4ex}\hrule}}}
\def\demend{\hfill $\videbox$\\}

\newcommand{\abs}[1]{\left\vert#1\right\vert}

\email{Bernard.Bercu@math.u-bordeaux.fr}
\email{Adrien.Richou@math.u-bordeaux.fr}

\begin{document}
\title[Large deviations for the Ornstein-Uhlenbeck process with shift]
{Large deviations for the Ornstein-Uhlenbeck process with shift \vspace{1ex}}
\author{Bernard Bercu}
\address{Universit\'e de Bordeaux, Institut de Math\'ematiques de Bordeaux,
UMR 5251, 351 Cours de la Lib\'eration, 33405 Talence cedex, France.}
\author{Adrien Richou}
\address{
}
\thanks{}

\begin{abstract}
We investigate the large deviation properties of the maximum likelihood estimators for
the Ornstein-Uhlenbeck process with shift. We propose a new approach to establish large deviation principles 
which allows us, via a suitable transformation, to circumvent the classical non-steepness problem.
We estimate simultaneously the drift and shift parameters.
On the one hand, we prove a large deviation principle for the maximum likelihood estimates 
of the drift and shift parameters. Surprisingly, we find that the drift estimator shares the
same large deviation principle as the one previously established for the Ornstein-Uhlenbeck process without shift.
Sharp large deviation principles are also provided. On the other hand, we show that the maximum likelihood 
estimator of the shift parameter satisfies a large deviation principle with a very unusual implicit rate function.
\end{abstract}

\maketitle

\vspace{-2ex}


\section{INTRODUCTION}

Consider the Ornstein-Uhlenbeck process with linear shift $\gamma \in \dR$, observed over the time 
interval $[0,T]$
\begin{equation}
\label{OUPS}
dX_t=\theta X_t dt +\gamma dt+dB_t
\end{equation}
where the drift parameter $\theta<0$, the initial state $X_0=0$ and the driven noise $(B_t)$ is a standard Brownian motion. 
This process is widely used in financial mathematics and it is known as the Vasicek model, see e.g. \cite{Jeanblanc09, V77}.
The maximum likelihood estimates of the unknown parameters $\theta$ and $\gamma$ are given by
\begin{equation}
\label{MLET}
\wh{\theta}_T=\frac{T \int_0^T X_t \,dX_t - X_T \int_0^T X_t \,dt }{T \int_0^T  X_t^2 \,dt - \Bigl(\int_0^T X_t \,dt\Bigr)^2}
\end{equation}
and
\begin{equation}
\label{MLEG}
\wh{\gamma}_T=\frac{X_T \int_0^T  X_t^2 \,dt -  \int_0^T X_t \,dX_t \int_0^T X_t \,dt}
{T \int_0^T  X_t^2 \,dt - \Bigl(\int_0^T X_t \,dt\Bigr)^2}.
\end{equation}
A wide range of literature is available on the asymptotic behavior of $(\wh{\theta}_T)$ and $(\wh{\gamma}_T)$.
It is well-known \cite{K04} that $\wh{\theta}_T$ and $\wh{\gamma}_T$ are both strongly consistent estimators of 
$\theta$ and $\gamma$ and their joint asymptotic normality is given by
\begin{equation}
\label{MLECLT}
\sqrt{T}\begin{pmatrix}
\ \wh{\theta}_{T} - \theta  \\
\ \wh{\gamma}_{T} - \gamma
\end{pmatrix}
\liml \cN(0,L)
\end{equation}
where the limiting matrix 
\begin{equation*}
L=2\begin{pmatrix}
\theta &  \gamma  \\
 \gamma &  \kappa
\end{pmatrix}
\end{equation*}
with $\kappa=(2\gamma^2 + \theta)/2\theta$. Moreover, concentration inequalities for $(\wh{\theta}_T)$ and $(\wh{\gamma}_T)$ 
and moderate deviations were established by Gao and Jiang \cite{GJ09}, while Jiang \cite{J12} recently obtained the joint law of iterated
logarithm as well as Berry-Esseen bounds. In the particular case $\gamma=0$, Florens-Landais and Pham \cite{FP99} proved the large
deviation principle (LDP) for $(\wh{\theta}_T)$, while sharp large deviation principles (SLDP) were established in \cite{BR01}. 
We also refer the reader to \cite{BCS12} for the sharp large deviations in the non-stationary case $\theta\geq 0$ and $\gamma=0$.

Our goal is to extend these investigations by establishing the large deviations properties of the maximum likelihood estimators
of the drift and shift parameters $\theta<0$ and  $\gamma$ in the situation where $\theta$ and $\gamma$
are estimated simultaneously. We shall propose a new approach to prove LDP which allows us, via a suitable transformation,
to circumvent the classical non-steepness problem. In particular, it could be possible to apply the same approach for
Jacobi or Cox-Ingersoll-Ross processes \cite{DZ09, Z02}.

The paper is organized as follows. In Section \ref{S-LDP}, we establish an LDP for the couple
$
\bigl(\wh{\theta}_T, \wh{\gamma}_T\bigr).
$
Via the contraction principle, one can realize that $(\wh{\theta}_T)$ shares the same LDP 
as the one previously established for the Ornstein-Uhlenbeck process without shift. 
One can also observe that $(\wh{\gamma}_T)$ satisfies an LDP with a very unusual implicit rate function. 
An SLDP for the sequence $(\wh{\theta}_T)$ is also provided. Section \ref{S-KSL} is devoted to three keystone lemmas which are at 
the core of our analysis. All the technical proofs of Sections \ref{S-LDP} and \ref{S-KSL} 
are postponed to Appendices A, B, and C.


\section{Large deviations results.}
\setcounter{equation}{0}
\label{S-LDP}

Our large deviations results are as follows.

\begin{thm}
\label{T-LDPG}
The couple $(\wh{\theta}_T,\wh{\gamma}_T)$ satisfies an LDP with good rate function
\begin{equation}
\label{RATEG}
 I_{\theta,\gamma}(c,d)= \left\{ \begin{array}{ccl}
             -{\displaystyle \frac{(\theta-c)^2}{4c}} + {\displaystyle \frac{1}{2}\left( \gamma -\frac{d \theta}{c}\right)^2} & \textrm{if} & {\displaystyle c \leq \frac{\theta}{3}},\\
             {\displaystyle (2c-\theta)}+ {\displaystyle \frac{1}{2}\left( \gamma -\frac{d \theta}{c}\right)^2}  & \textrm{if} & {\displaystyle c \geq \frac{\theta}{3}} \textrm{ and } {\displaystyle c \neq 0},\vspace{1ex}\\
	     {\displaystyle -\theta}  & \textrm{if} &  (c,d)=(0,0),\vspace{2ex}\\
	     {\displaystyle +\infty}  &  \textrm{if} & c=0  \textrm{ and } d \neq 0.
            \end{array} \right.
\end{equation}
\end{thm}

A direct application of the contraction principle \cite{DZ98} immediately leads to the two following corollaries.

\setcounter{cor}{1}

\begin{cor}
\label{C-LDPT}
 The sequence $(\wh{\theta}_T)$ satisfies an LDP with good rate function
\begin{equation}
\label{RATET}
 I_{\theta}(c)= \left\{ \begin{array}{ccc}
             -{\displaystyle \frac{(c-\theta)^2}{4c}} & \textrm{if} & {\displaystyle c \leq \frac{\theta}{3}},\vspace{1ex}\\
             2c-\theta & \textrm{if} &  {\displaystyle c\geq \frac{\theta}{3}}.
            \end{array} \right.
\end{equation}
\end{cor}

\begin{cor}
\label{C-LDPG}
 The sequence $(\wh{\gamma}_T)$ satisfies an LDP with good rate function
\begin{equation}
\label{RATEGG}
 I_{\gamma}(d)= \inf \Bigl\{ I_{\theta,\gamma}(c,d) \ \slash \ c \in \dR \Bigr\}.
\end{equation}
\end{cor}

\begin{proof}
The proofs are given is Section \ref{S-PLDP}.
\end{proof}

\setcounter{rem}{3}

\begin{rem}
On the one hand, one can observe that $(\wh{\theta}_T)$ shares exactly the same LDP 
as the one previously established by Florens-Landais and Pham \cite{FP99} for the Ornstein-Uhlenbeck without shift $\gamma=0$.
On the other hand, $(\wh{\gamma}_T)$ satisfies an LDP with a very unusual rate function. Unfortunately, an explicit expression
for this rate function is quite complicated. Its very particular form is given in the special cases
$(\theta,\gamma)=(-2, 2)$ and $(\theta,\gamma)=-(2,1)$ in Figure \ref{FIGRF} below.
\end{rem}

\begin{figure}[htb]
\vspace{-6cm}
\begin{minipage}[b] {1\linewidth}
\centering 
\centerline {\includegraphics[width=10cm,height=18cm,angle=0]{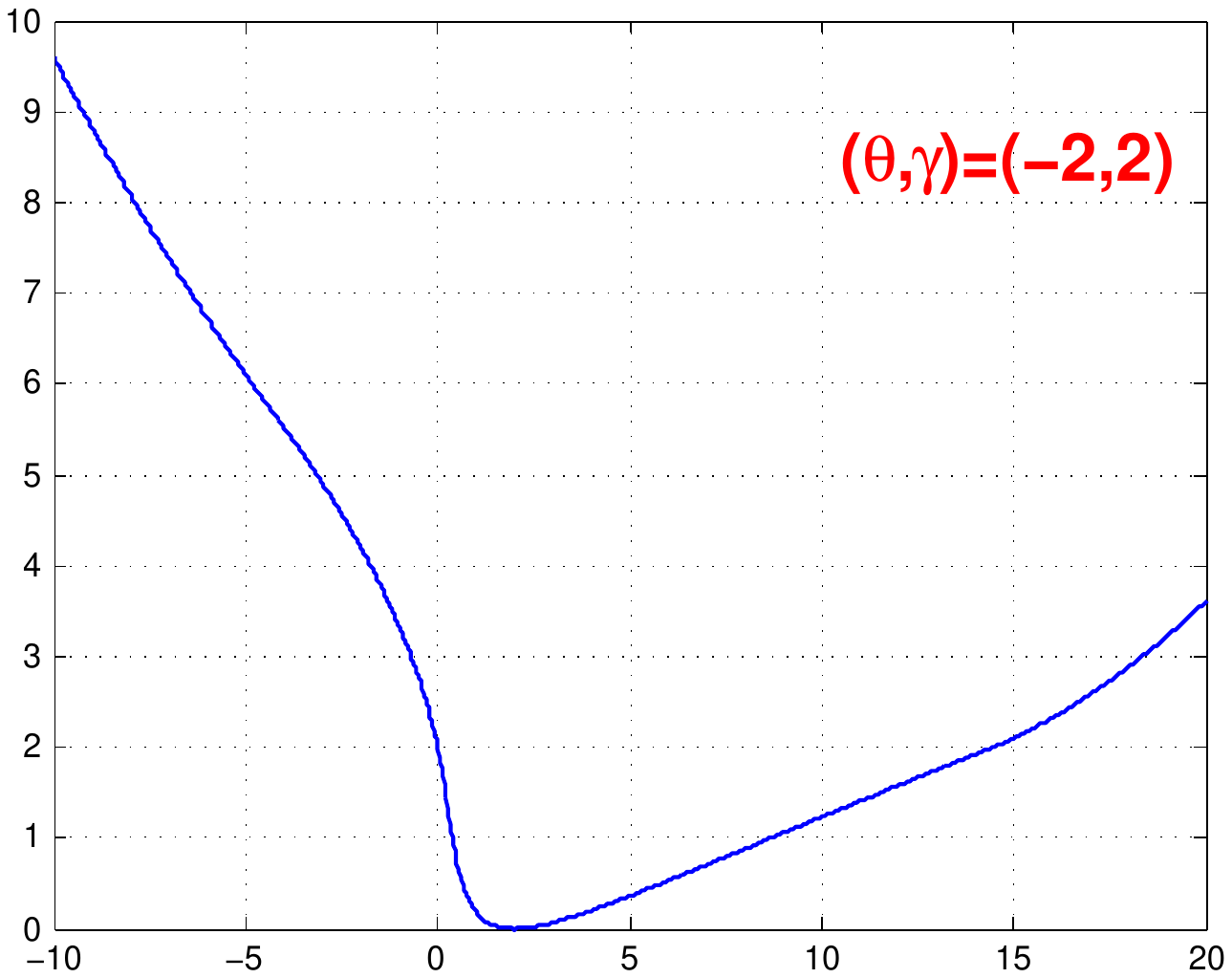} 
\hspace{-4cm}
\includegraphics[width=10cm,height=18cm,angle=0]{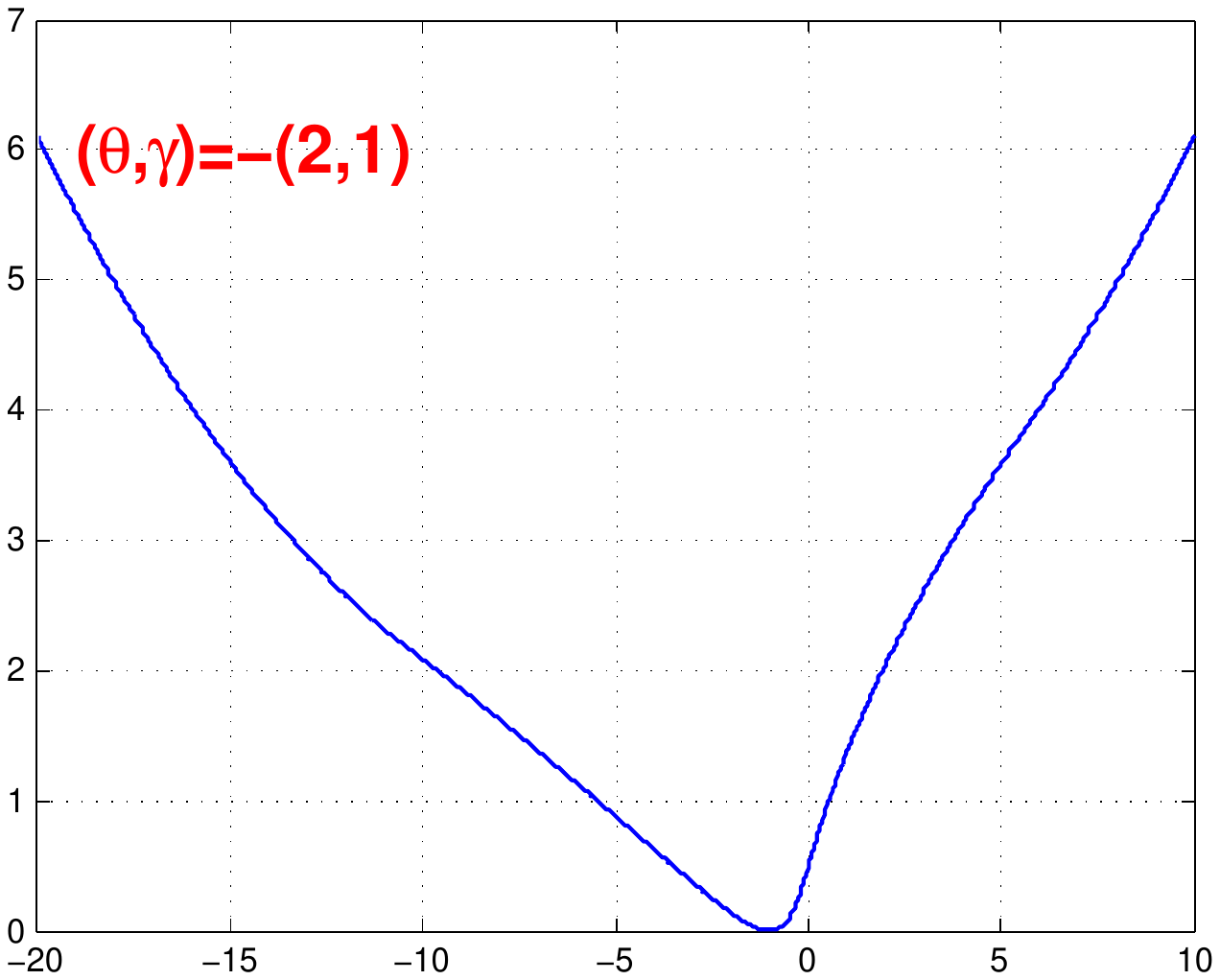} }
\vspace{0.1cm}
\medskip
\end{minipage}
\vspace{-7cm}
\caption{Rate functions for the drift parameter.}
\label{FIGRF}
\end{figure}

\noindent
Our goal is now to improve Corollary \ref{C-LDPT} by a first order SLDP for $(\wh{\theta}_T)$. 
It is of course possible to establish SLDP of any order for the sequence $(\wh{\theta}_T)$. However, for clearness sake, 
we have chosen to restrict ourself to a first order expansion.

\begin{thm}
\label{T-SLDPT}
Consider the Ornstein-Uhlenbeck process with shift given by \eqref{OUPS} where the drift parameter $\theta<0$.
\begin{enumerate}
\item[a)] For all $\theta < c < \theta/3$, we have for $T$ large enough,
\begin{equation}
\label{E-SLDPT1}
\dP(\wh{\theta}_T \geq c) =\frac{e^{-TI_{\theta}(c)+J(c)}}{a_c \sigma_c  \sqrt{2\pi T}} \Bigl(1+o(1)\Bigr)
\end{equation}
while for $c<\theta$,
\begin{equation}
\label{E-SLDPT2}
   \dP(\wh{\theta}_T \leq c) =-\frac{e^{-TI_{\theta}(c)+J(c)}}{a_c \sigma_c  \sqrt{2\pi T}} \Bigl(1+o(1)\Bigr)
  \end{equation}
where 
\begin{equation}
\label{DEFACSIGC1}
   a_c=\frac{c^2-\theta^2}{2c} \hspace{1cm} \textrm{and} \hspace{1cm} \sigma_c^2=-\frac{1}{2c},
\end{equation}
\begin{equation}
\label{DEFJ}
J(c)=-\frac{1}{2}\log \left( \frac{\theta^2(c+\theta)(3c-\theta)}{4c^4} \right) +\gamma^2 \frac{(c-\theta)^2}{4c\theta^2}.
\end{equation}
\item[b)] For all $c>\theta/3$ with $c \neq 0$, we have for $T$ large enough,
\begin{equation}
\label{E-SLDPT3}
 \dP(\wh{\theta}_T \geq c) =\frac{e^{-TI_{\theta}(c)+K(c)}}{ a_c \sigma_c\sqrt{2\pi T}} \Bigl(1+o(1)\Bigr)
 \end{equation}
   where
   \begin{equation}
   \label{DEFACSIGC2}
   a_c=2(c-\theta) \hspace{1cm} \textrm{and} \hspace{1cm} \sigma_c^2=\frac{c^2}{2(2c-\theta)^3}.
  \end{equation} 
  \begin{equation}
   \label{DEFK}
   K(c)=-\frac{1}{2}\log \left( \frac{\theta^2(c-\theta)(3c-\theta)}{4c^2(2c-\theta)^2} \right) -\frac{\gamma^2}{\theta^2} (2c-\theta).
  \end{equation}
\item[c)] For $c=\theta/3$, we have for $T$ large enough,
\begin{equation}
\label{E-SLDPT4}
   \dP(\wh{\theta}_T \geq c) =\frac{e^{-TI_{\theta}(c)+ \gamma^2b_{\theta}}}{6 \pi T^{1/4}}
\frac{\Gamma(1/4)}{\sqrt{2}a_{\theta}^{3/4}\sigma_{\theta}} \Bigl(1+o(1)\Bigr)
\end{equation}
where
\begin{equation}
\label{DEFACSIGC4}
a_{\theta}= -\frac{4\theta}{3},
\hspace{1cm}b_{\theta}= \frac{1}{3\theta}, \hspace{1cm}
\sigma_{\theta}^2=-\frac{3}{2\theta}.
\end{equation}
\item[d)] For $c=0$, we have for $T$ large enough,
\begin{equation}
\label{E-SLDPT5}
   \dP(\wh{\theta}_T \geq c) =\frac{\sqrt{2}e^{-TI_{\theta}(c)+\frac{\gamma^2}{\theta}+2}}{\sqrt{2\pi T}\sqrt{-\theta}} \Bigl(1+o(1)\Bigr).
\end{equation}

 \end{enumerate}
\end{thm}


\section{Three keystone lemmas.}
\setcounter{equation}{0}
\label{S-KSL}

First of all, let us recall some elementary properties of the Ornstein-Uhlenbeck process with linear shift \cite{Jeanblanc09}, \cite{K04}.
One can observe that the process $(X_T)$ can be rewritten as
$X_T=Y_T+ m_T$ where
$$
m_T=\dE[X_T]=-\frac{\gamma}{\theta}(1-e^{\theta T})
$$
and $(Y_T)$ is the Ornstein-Uhlenbeck process without shift
$$
Y_T= e^{\theta T}\int_0^T e^{-\theta t}dB_t.
$$
By the same token, if
$$
\overline{X}_T= \frac{1}{T} \int_0^T X_t dt
\hspace{1cm} \text{and}\hspace{1cm} 
\overline{Y}_T= \frac{1}{T} \int_0^T Y_t dt,
$$
we clearly have
$\overline{X}_T=\overline{Y}_T+ \mu_T$
where
$$
\mu_T=\dE[\overline{X}_T]=-\frac{\gamma}{\theta}\Bigl(1+\frac{1}{\theta T}(1-e^{\theta T}) \Bigr).
$$
Therefore, the random vector
$$
\begin{pmatrix}
X_T \\
\overline{X}_T
\end{pmatrix} \sim \cN \left( \begin{pmatrix}
m_T \\
\mu_T
\end{pmatrix}\!, \ \Gamma_T(\theta) \right)
$$
where the covariance matrix $\Gamma_T(\theta)$ is given by 
\begin{equation}
\label{DEFGAMT}
\Gamma_T(\theta)= 
\begin{pmatrix}
a_T(\theta) & b_T(\theta) \vspace{1ex} \\
b_T(\theta) & c_T(\theta)
\end{pmatrix}
\end{equation}
with 
\begin{eqnarray*}
a_T(\theta) &\!\!=\!\!& \frac{1}{2\theta} \Bigl( e^{2 \theta T} - 1 \Bigr), \hspace{0.5cm}
b_T (\theta)=   \frac{1}{2\theta^2 T} \Bigl( e^{\theta T} - 1 \Bigr)^2, \\
c_T(\theta) &\!\!=\!\!& \frac{1}{\theta^2 T^2} \Bigl( \frac{1}{2\theta} \Bigl( e^{2 \theta T} - 1 \Bigr) - \frac{2}{\theta} \Bigl( e^{\theta T} - 1\Bigr) +T\Bigr).
\end{eqnarray*}
Denote by $\Lambda_T$ the normalized cumulant generating function of the triplet
$$\left( \frac{X_T}{\sqrt{T}}, \frac{1}{T}\int_0^T X_t^2dt, \frac{1}{T} \int_0^T X_t dt\right)$$
defined, for all $(a,b,c) \in \dR^3$, by
$$\Lambda_T(a,b,c) = \frac{1}{T} \log \dE \left[\exp \left( a\sqrt{T} X_T + b \int_0^T X_t^2 dt + c \int_0^T X_t dt \right) \right].$$
Our first lemma deals with the extended real function $\Lambda$ defined as the pointwise limit of $\Lambda_T$.
\begin{lem}
\label{L-LEMFOND1}
Let $\cD_\Lambda$ be the effective domain of $\Lambda$
$$\cD_\Lambda = \Bigl\{ (a,b,c) \in \dR^3 \ \slash \ b < \theta^2/2 \Bigr\}$$
and set $\vp(b) = \sqrt{\theta^2 -2b}$. Then, for all $(a,b,c) \in \cD_\Lambda$, we have
\begin{equation}
\label{DEFLL}
\Lambda(a,b,c) = -\frac{1}{2}\left(\theta  + \vp(b)+ \gamma^2\right)
+\frac{1}{2}\left(\frac{a^2}{\vp(b) - \theta}\right)+\frac{1}{2} \left( \frac{c-\theta \gamma}{\vp(b)}\right)^2.
\end{equation}
\end{lem}
\begin{proof}
The proof in given in Appendix A.
\end{proof}

A direct calculation shows that the function $\Lambda$ is steep, which means that the norm of its gradient
goes to infinity for any sequence in the interior of $\cD_\Lambda$ converging to a boundary point of $\cD_\Lambda$.  
This is the reason why we are able to deduce an LDP for the couple $(\wh{\theta}_T, \wh{\gamma}_T)$.
In order to establish the SLDP for the drift parameter $\wh{\theta}_T$, it is necessary to modify our strategy.
To be more precise, we shall now focus our attention on the normalized cumulant generating function $\cL_T$ of the couple
$$\left(\frac{1}{T}\int_0^T (X_t - \overline{X}_T) dX_t, \frac{1}{T}S_T\right)$$
where
$$
S_T= \int_0^T (X_t - \overline{X}_T)^2dt,
\vspace{2ex}
$$
which is given, for all $(a,b) \in \dR^2$, by
$$\cL_T(a,b) = \frac{1}{T} \log \dE \left[\exp \left( a\int_0^T (X_t - \overline{X}_T) dX_t + b S_T \right) \right].$$
\  \\
The reason for this is twofold. On the one hand, it is not possible to deduce an SLDP for $(\wh{\theta}_T)$ via $\Lambda_T$.
On the other hand, it immediately follows from \eqref{MLET} that
\begin{equation}
\label{EQMLET}
\wh{\theta}_T =\frac{\int_0^T (X_t - \overline{X}_T)dX_t}{S_T}.
\end{equation}
However, one can observe that for all $c \in \dR$, $\dP(\wh{\theta}_T\geq c)=\dP(\cZ_T(1,-c) \geq 0)$ where,
for all $(a,b) \in \dR^2$, $\cZ_T(a,b)$ stands for the random variable 
\begin{equation}
\label{DEFZTAB}
\cZ_T(a,b) = a\int_0^T (X_t - \overline{X}_T) dX_t +b S_T.
\end{equation}
Our second lemma provides the full asymptotic expansion for $\cL_T$.
Denote by $\cL$ the extended real function defined as the pointwise limit of $\cL_T$.

\begin{lem} 
\label{L-LEMFOND2}
Let $\cD_\cL$ be the effective domain of $\cL$
$$
\cD_\cL= \Bigl\{ (a,b) \in \dR^2  \ \slash \ \theta^2 -2b > 0 \quad \textrm{and} \quad a + \theta < \sqrt{\theta^2-2b} \Bigr\}
$$
and set $\vp(b) = \sqrt{\theta^2 -2b}$, $\tau(a,b)=\vp(b)-(a+\theta)$. Then, for all $(a,b) \in \cD_\cL$ and $T$ large enough, we have 
\begin{equation} 
\label{E-DECFOND}
\cL_T(a,b)=\cL(a,b) + \frac{1}{T} \cH(a,b) + \frac{1}{T^2} \cR_T(a,b)
\end{equation}
where
\begin{eqnarray}
\cL(a,b)   &=& -\frac{1}{2} \left(a + \theta + \sqrt{\theta^2-2b} \right),\label{E-DEFL}\\
\cH(a,b)   &=& -\frac{1}{2} \log \left( \frac{\tau(a,b) \theta^2}{2\vp^3(b)} \right)-\frac{\gamma^2(a + \theta + \sqrt{\theta^2-2b})}{2\theta^2}.
\label{E-DEFH}
\end{eqnarray}
\ \vspace{1ex} \\
Moreover, the remainder $\cR_T(a,b)$ may be explicitly calculated as a rational function of $a$, $b$, $T$ and $\exp(-\vp(b)T)$.
 In addition, $\cR_T$ can be extended to the two-dimensional complex plane and it is a bounded analytic function as soon as the real parts of its arguments belong to the interior of $\cD_\cL$.
%
\end{lem}
\begin{proof}
The proof in given in Appendix B.
\end{proof}
\noindent
Our third lemma relies on the Karhunen-Lo\`eve expansion of the process $(X_T)$.
Denote by $\cF$ the class of all real-valued continuous functions $f$ such that
$f(x)=x^2 h(x)$ where $h$ is continuous. Moreover, let
$g$ be the spectral density of the stationary Ornstein-Uhlenbeck process
without shift $\gamma=0$ given, for all $x \in \dR$, by
\begin{equation}
\label{SPECDENS}
g(x)=\frac{1}{\theta^2 + x^2}.
\end{equation}

\vspace{1ex}

\begin{lem}
\label{L-LEMDEC}
One can find two sequences of real numbers $(\alpha_k^{T} )$ and $(\beta_k^{T} )$ both in $\ell^2(\dN)$
such that
\begin{equation}
\label{DECZTAB}
\cZ_T(a,b) = \dE[\cZ_T(a,b)] + \sum_{k=1}^{\infty} \alpha_k ^{T}\left(\varepsilon_k^2-1\right) +\beta_k ^{T}\varepsilon_k 
\end{equation}
where $(\varepsilon_k)$ are independent standard $\cN(0,1)$ random variables. Moreover, for all $(a,b) \in \cD_\cL$, there exist two constants 
$A>0$ and $B>0$ that do not depend on $T$ such that, for $T$ large enough, $\alpha_k^T \in [-A,A]$ for all $k \geq 1$ and 
\begin{equation}
\label{BOUNDBETAT}
\sum_{k=1}^{+\infty} (\beta_k^T)^2 <B.
\end{equation}
Consequently, for all $(a,b) \in \cD_\cL$ and $x\in \dR$ such that $\abs{x} <1/2A$, and for $T$ large enough, we have
\begin{equation*} 
\cL_T(xa,xb)=\frac{1}{T} \dE[\cZ_T(xa,xb)] -\frac{1}{2T} \sum_{k=1}^{\infty}\Bigl(\log(1-2x\alpha_k ^{T})+2x\alpha_k^T\Bigr)
+\frac{1}{2T} \sum_{k=1}^{\infty}\frac{(x\beta_k ^{T})^2}{1-2x\alpha_k ^{T}}.
\end{equation*}
Finally, if $b \neq 0$, the empirical spectral measure
\begin{equation} 
\label{L-DEFNUT}
\nu_T=\frac{1}{T} \sum_{k=1}^{\infty} \delta_{\alpha_k ^{T}}
\end{equation}
satisfies, for any continuous $f\in \cF$ with compact support
\begin{equation} 
\label{L-LIMNUT}
\lim_{T \rightarrow \infty} \langle \nu_T,f \rangle
= \lim_{T \rightarrow \infty}
\frac{1}{T} \sum_{k=1}^{\infty} f(\alpha_k ^{T})
= \langle \nu ,f \rangle
= \frac{1}{2\pi} \int_{\mathbb{R}} f(bg(x))dx.
\end{equation}
\end{lem}

\begin{proof}
The proof is given in Appendix C.
\end{proof}


\section{Proofs of the large deviations results.}
\setcounter{equation}{0}
\label{S-PLDP}


First of all, 
\begin{equation*}
\wh{V}_T=
\begin{pmatrix}
 \wh{\theta}_T \\
\wh{\gamma}_T
\end{pmatrix}
\hspace{1cm}\text{and}\hspace{1cm}
\wt{V}_T=
\begin{pmatrix}
 \wt{\theta}_T \\
\wt{\gamma}_T
\end{pmatrix}
\end{equation*}
where
\begin{equation}
\label{MLEQUE}
\wt{\theta}_T=\frac{\int_0^T X_t dX_t}{S_T} 
\hspace{1cm}\text{and}\hspace{1cm}
\wt{\gamma}_T= - \wt{\theta}_T \overline{X}_T.
\end{equation}
The following lemma shows that the sequences $(\wh{V}_T)$  and $(\wt{V}_T)$ share the same LDP. We refer the reader to
\cite{DZ98} for the classical notion of exponential approximation.

\begin{lem}
 \label{L-EXPEQUE}
The sequences of random vectors $(\wh{V}_T)$  and $(\wt{V}_T)$ are exponentially equivalent, that is to say, for all $\varepsilon >0$,
\begin{equation}
\lim_{T \rightarrow +\infty} \frac{1}{T} \log \dP\left( \parallel \wh{V}_T-\wt{V}_T \parallel > \varepsilon \right)=-\infty.
\end{equation}
In particular, if $(\wt{V}_T)$ satisfies an LDP with good rate function $I$, then  the same LDP holds for 
$(\wh{V}_T)$. 
\end{lem}

\begin{proof}
It is easy to see from the very definition of our estimates
given in \eqref{MLET}, \eqref{MLEG} and \eqref{MLEQUE} that
\begin{equation*}
\wh{\theta}_T-\wt{\theta}_T=\frac{X_T}{T}\left(\frac{\overline{X}_T}{\Sigma_T}\right)
\hspace{1cm}\text{and}\hspace{1cm}
\wh{\gamma}_T-\wt{\gamma}_T=\frac{X_T}{T}\left(1-\frac{(\overline{X}_T)^2}{\Sigma_T}\right)
\end{equation*}
where
$\Sigma_T=S_T/T$.
On the event $\{ | \overline{X}_T | \leq \xi, \ \Sigma_T \geq \xi^{-1} \}$ with $\xi >1$, we have
$$
\parallel \wh{V}_T-\wt{V}_T \parallel  \leq \sqrt{3} \xi^3 \frac{[X_T|}{T}.
$$
Hence, for all $\varepsilon>0$,
\begin{equation}
\label{EQUI1}
\dP\left( \parallel \wh{V}_T-\wt{V}_T \parallel > \varepsilon \right) \leq 
\dP\Bigl( \frac{[X_T|}{T} \geq  \frac{\varepsilon \xi^{-3}}{\sqrt{3}} \Bigr) +
\dP(  | \overline{X}_T | \geq \xi ) + \dP( \Sigma_T \leq \xi^{-1} ).
\end{equation}
On the one hand, it is not hard to see that for all $c>0$,
\begin{equation}
\label{EQUI2}
\lim_{T \rightarrow \infty}\frac{1}{T}\log \dP \left( \frac{|X_T|}{T}  \geq c \right) = -\infty.
\end{equation}
As a matter of fact, we recall that $X_T$ is a Gaussian $\cN(m_T,a_T(\theta))$ random variable.
Consequently, for all $c>0$,
$$
\lim_{T \rightarrow \infty}\frac{1}{T^2}\log\dP \Bigl( \abs{X_T} \geq cT \Bigr)=\theta c^2
$$
which immediately leads to \eqref{EQUI2}, as $\theta <0$. By the same token, we already saw at the beginning of Section
\ref{S-KSL}
that $\overline{X}_T$ is a Gaussian $\cN(\mu_T,c_T(\theta))$ random variable.
It implies that for all $c>0$ such that $ c>|\gamma|/|\theta|$,
\begin{equation}
\label{ineg1}
\limsup_{T \rightarrow + \infty} \frac{1}{T}\log \dP \Bigl( \abs{\overline{X}_T} \geq c \Bigr) \leq 
-\frac{\theta^2}{2}  \left(  c+\abs{\frac{\gamma}{\theta}}\right)^2.
\end{equation}
On the other hand, we immediately deduce from Lemma \ref{L-LEMFOND2} together with
G\"artner-Ellis's theorem that the sequence $(\Sigma_T)$ satisfies an LDP with speed $T$ and good rate function
\begin{equation*}
I(c)= \left\{ \begin{array}{ccc}
{\displaystyle \frac{(2\theta c+1)^2}{8c}} & \textrm{if} & {\displaystyle c>0},\vspace{1ex}\\
 +\infty & \textrm{if} & {\displaystyle c\leq 0}.           
            \end{array} \right.
\end{equation*}
Therefore, for all $c>0$ such that $-2\theta c <1$,
\begin{equation}
\label{EQUI3}
\lim_{T \rightarrow \infty}\frac{1}{T}\log \dP \Bigl( \Sigma_T \leq c \Bigr) = - \frac{(2\theta c+1)^2}{8c}.
\end{equation}
Finally, it follows from the conjunction of \eqref{EQUI1}, \eqref{EQUI2}, \eqref{ineg1}, and \eqref{EQUI3} that, for all
$\veps >0$ and for $\xi >1$ large enough,
\begin{equation}
\label{EQUI4}
\limsup_{T \rightarrow + \infty} \frac{1}{T}\log \dP \left( \parallel \wh{V}_T-\wt{V}_T \parallel > \varepsilon \right) 
\leq - M_{\theta,\gamma}(\veps,\xi)
\end{equation}
where
\begin{equation*}
M_{\theta,\gamma}(\veps,\xi)  =  \min \left( \frac{\theta^2}{2} \Bigl( \xi + \abs{\frac{\gamma}{\theta}}\Bigr)^2\!\!,
 \frac{\xi}{8} \Bigl( \frac{2\theta}{\xi}  + 1\Bigr)^2 \right).
\end{equation*}
One can observe that if $\xi$ goes to infinity, $M_{\theta,\gamma}(\veps,\xi)$ tend to infinity, which is
exactly what we wanted to prove.
\end{proof}

We are now in the position to prove our LDP results. Our strategy is to establish an LDP for the triplet
\begin{equation*}
\left( \frac{X_T}{\sqrt{T}}, \frac{1}{T}\int_0^T X_t^2dt, \frac{1}{T} \int_0^T X_t dt\right)
\end{equation*}
and then to make use of the contraction principle \cite{DZ98} in order to prove Theorem \ref{T-LDPG}
via Lemma \ref{L-EXPEQUE}. The limiting cumulant generating function $\Lambda$ of the above triplet was
already calculated in Lemma \ref{L-LEMFOND1}. It is not difficult to check that the function $\Lambda$ is steep 
on its effective domain $\cD_\Lambda$. Consequently, we deduce from G\"artner-Ellis's theorem that the
above triplet satisfies an LDP with good rate function $I$ given by the Fenchel-Legendre transform of $\Lambda$,
$$I(\lambda,\mu,\delta)= \sup_{(a,b,c) \in \cD_\Lambda} \Bigl\{\lambda a +\mu b +\delta c - \Lambda(a,b,c) \Bigr\}.$$
We can prove after some straightforward calculations that
\begin{equation}
\label{RATETRIP}
 I(\lambda,\mu,\delta)= \left\{ \begin{array}{ccc}
{\displaystyle  \frac{\theta^2\mu- \theta \lambda^2}{2} + \frac{\theta+\gamma^2+2 \theta \gamma \delta}{2}
+\frac{(1+\lambda^2)^2}{8(\mu-\delta^2)}} & \textrm{if} & {\displaystyle \delta^2<\mu},\vspace{1ex}\\
             +\infty & \textrm{if} & \delta^2 \geq \mu.
            \end{array} \right.
\end{equation}
Hereafter, it follows from the well-known It\^o's formula that
\begin{equation}
\label{ITO}
\int_0^{T} X_t dX_t= \frac{1}{2} \Bigl( X_{T}^2-T\Bigr)
\end{equation}
which implies that
$$\wt{V}_T=f\left(\frac{X_T}{\sqrt{T}},\frac{1}{T}\int_0^T X_t^2dt,\frac{1}{T}\int_0^T X_t dt \right)$$
where $f$ is the continuous function given, for all $(\lambda,\mu,\delta) \in \dR^3$ such that $\mu > \delta^2$, by
\begin{equation*}
 f(\lambda,\mu,\delta)=
\begin{pmatrix}
{\displaystyle \frac{\lambda^2-1}{2(\mu-\delta^2)}} \\
{\displaystyle-\frac{\delta(\lambda^2-1)}{2(\mu-\delta^2)}}
\end{pmatrix}
.
\end{equation*} 
Therefore, we infer from the contraction principle given e.g. by Theorem 4.2.1 of \cite{DZ98}, together  with Lemma \ref{L-EXPEQUE},
that the sequences of random vectors $(\wh{V}_T)$ and $(\wt{V}_T)$ share the same LDP with good rate function 
\begin{equation*}
 I_{(\theta,\gamma)}(c,d)=\inf \Bigl\{ I(\lambda,\mu,\delta)  \ \slash \ (\lambda,\mu,\delta) \in \mathbb{R}^3, \mu>\delta^2, 
 f(\lambda,\mu,\delta)= \begin{pmatrix} c\\d \end{pmatrix}
 \Bigr\},
\end{equation*}
where the infimum over the empty set is equal to $+\infty$.  
Finally, we obtain the rate function given by \eqref{RATEG} thanks to elementary calculations, which completes the proof of Theorem \ref{T-LDPG}.
\demend


\section{Proofs of the sharp large deviations results.}
\setcounter{equation}{0}
\label{S-PSLDP}


\subsection{Proof of Theorem \ref{T-SLDPT}, first part}
First of all, it follows from straightforward calculations that the effective domain
$\cD_\cL$ given in Lemma \ref{L-LEMFOND2} can be rewritten as
\begin{equation*}
\cD_\cL= \left\{ \begin{array}{ccc}
           {\displaystyle  \Bigl]-\infty,\frac{-\theta^2}{2c} \Bigr[ }& \textrm{if} &  {\displaystyle   \theta < c\leq \frac{\theta}{2}},\\
              {\displaystyle  \Bigl]-\infty,2(c -\theta) \Bigr[ }& \textrm{if} & {\displaystyle\frac{\theta}{2} < c \leq 0},\\
             {\displaystyle  \Bigl]\frac{-\theta^2}{2c}, 2(c -\theta) \Bigr[ } & \textrm{if} & c>0.
            \end{array} \right.
\end{equation*}
In addition, for all $a \in \cD_\cL$, let $\vp(a)=\sqrt{\theta^2+2ac}$ and denote
\begin{eqnarray}
 L(a)&=& \cL(a,-ac)   = -\frac{1}{2} \left(a + \theta + \sqrt{\theta^2+2ac} \right),\label{E-DEFL2}\\
 H(a)&=& \cH(a,-ac)   = -\frac{1}{2} \log \left( \frac{(\vp(a)-a-\theta) \theta^2}{2\vp^3(a)} \right)-\frac{\gamma^2(a + \theta + \vp(a))}{2\theta^2}\label{E-DEFH2}.
\end{eqnarray}
The function $L$ is not steep as the derivative of $L$ is finite at the boundary of $\cD_\cL$. Moreover, $L^\prime(a)=0$ if and only if 
$a=a_c$ with $a_c$ given by \eqref{DEFACSIGC1}.
Finally, one can observe that $a_c\in \cD_\cL$ only if $c<\theta/3$. We shall focus our attention on the SLDP 
in the easy case $\theta < c < \theta/3$. Denote by $L_T$ the normalized 
cumulant generating function of the random variable $Z_T(a)=\cZ_T(a,-ca)$. 
We can split $\dP(\wh{\theta}_T \geq c)$ into two terms, $\dP(\wh{\theta}_T \geq c)=A_TB_T$ with
\begin{eqnarray}
\label{DEFAT} A_T &=& \exp (TL_T(a_c)),\\
\label{DEFBT} B_T &=& \dE_T\left[\exp (- Z_T(a_c))\rI_{Z_T(a_c) \geq 0}\right],
\end{eqnarray}
where $\dE_T$ stands for the expectation after the usual change of probability
\begin{equation}
\label{DEFET}
 \frac{\dd \dP_T}{\dd \dP} = \exp \left( Z_T(a_c) - TL_T(a_c)\right).
\end{equation}
On the one hand, we can deduce from Lemma \ref{L-LEMFOND2} that
\begin{equation}
\label{DEVAT}
 A_T = \exp( TL(a_c)+H(a_c))\Bigl(1+o(1)\Bigr)=\exp( -TI_{\theta}(c)+J(c))\Bigl(1+o(1)\Bigr).
\end{equation}
It remains to establish an asymptotic expansion for $B_T$ which can be rewritten as
\begin{equation}
\label{NBT}
 B_T=\mathbb{E}_T \Bigl[\exp (-a_c\sigma_c\sqrt{T} U_T)\rI_{U_T \geq 0}\Bigr]
 \hspace{0.5cm} \text{where} \hspace{0.5cm}
 U_T=\frac{Z_T(1)}{\sigma_c \sqrt{T}}.
\end{equation}

\begin{lem}
\label{L-BT}
 For all $\theta < c < \theta/3$, we have 
\begin{equation}
\label{DEVBT}
 B_T=\frac{1}{a_c \sigma_c \sqrt{2\pi T}}\Bigl(1+o(1)\Bigr).
\end{equation}
\end{lem}

\begin{proof}
Denote by $\Phi_T$ the characteristic function of $U_T$ under $\mathbb{P}_T$. As $\theta < c < \theta/3$, it follows 
from \eqref{DEFACSIGC1} that $a_c>0$ and $\sigma_c>0$. Moreover, \eqref{DEFET} immediately implies that
\begin{equation}
\label{PHIT}
\Phi_T(u)=\exp \left(T L_T\left(a_c+\frac{iu}{\sigma_c\sqrt{T}}\right)-TL_T(a_c)\right).
\end{equation}
First of all, we deduce from Lemma \ref{L-LEMDEC} that for $T$ large enough, $\Phi_T$ belongs to $L^2(\mathbb{R})$. As a matter of fact,
as soon as $1-2a_c \alpha_k ^{T}>0$ for all $k\geq 1$ and $T$ large enough, we obtain from Lemma \ref{L-LEMDEC} and \eqref{PHIT} that
\begin{eqnarray}
|\Phi_T(u)|^2 &=&\prod_{k=1}^\infty \left( 1 + \frac{4 u^2 (\alpha_k ^{T})^2}{\sigma_c^2 T(1-2a_c \alpha_k ^{T})^2} \right)^{-1/2}
\exp\left(- \frac{ (a_c \beta_k ^{T})^2}{1-2a_c \alpha_k ^{T}} \right) \nonumber, \\
& \leq & \prod_{k=1}^\infty \left( 1 + \frac{4 u^2 (\alpha_k ^{T})^2}{\sigma_c^2 T(1-2a_c \alpha_k ^{T})^2} \right)^{-1/2}.
\label{PHITSQ}
\end{eqnarray}
For all $\veps>0$ small enough such that $1-2a_c \veps >0$ we denote
$$
q_T(\veps)=\sum_{k=1}^\infty \rI_{\left|\alpha_k ^{T}\right|>\veps}.
$$
It follows from Lemma \ref{L-LEMDEC} that it exists some constants $\veps >0$ and $\eta>0$, depending only on $\veps$, such that
\begin{equation}
\label{LIMQT}
\liminf_{T\rightarrow \infty} \frac{q_T(\veps)}{T} \geq 2 \eta.
\end{equation}
Hence, we infer from \eqref{PHITSQ} and \eqref{LIMQT} that for $T$ large enough,
\begin{equation}
|\Phi_T(u)|^2 \leq \left( 1 + \frac{\xi^2 u^2}{T} \right)^{-\eta T}
\hspace{1cm} \text{where} \hspace{1cm}
\xi=\frac{2 \veps}{\sigma_c(1+2 a_c \veps)}
\end{equation}
which clearly ensures, whenever $\eta T \geq 1$, that 
\begin{equation}
\label{MAJPHIT}
|\Phi_T(u)|^2 \leq \frac{1}{1+ \eta \xi^2 u^2}.
\end{equation}
Consequently, we find from \eqref{MAJPHIT} that for $T$ large enough, $\Phi_T$ belongs to $L^2(\mathbb{R})$.
Therefore, we obtain from Parseval's formula that $B_T$, given by \eqref{NBT}, can be rewritten as
\begin{equation}
 \label{NNBT}
 B_T=\frac{1}{2\pi a_c \sigma_c \sqrt{T}} \int_{\mathbb{R}}\left(1+\frac{iu}{a_c\sigma_c\sqrt{T}}\right)^{-1}\Phi_T(u)du.
\end{equation}
However, we deduce from Lemma \ref{L-LEMFOND2} that, for all $u \in \dR$,
\begin{equation}
\lim_{T \rightarrow \infty}  \Phi_T(u) =   \lim_{T \rightarrow \infty} \exp \left(T L_T\left(a_c+\frac{iu}{\sigma_c\sqrt{T}}\right)-TL_T(a_c)\right)=\exp \left( -\frac{u^2}{2} \right)
 \label{CVGPHIT}
\end{equation}
as $L^{\prime \prime}(a_c)=\sigma_c^2$, which means that the distribution of $U_T$ under $\dP_T$ converges to the standard
$\cN(0,1)$ distribution. Finally, \eqref{DEVBT} follows from \eqref{NNBT}, \eqref{CVGPHIT} and the
Lebesgue dominated convergence theorem.
\end{proof}

\noindent
Finally, the expansion \eqref{E-SLDPT1} immediately follows from the conjunction of \eqref{DEVAT} and \eqref{DEVBT}. 
The proof of \eqref{E-SLDPT2} follows  exactly the same lines, the only notable point to mention being a 
change of sign in Parseval's formula.
\demend

\subsection{Proof of Theorem \ref{T-SLDPT}, second part}
We shall now proceed to the proof of the SLDP in the more complex case $c> \theta/3$ with $c\neq 0$. One can easily check that the function $L$,
given by \eqref{E-DEFL2}, is decreasing and reaches its minimum at the right boundary point $a_c=2(c-\theta)$ of the domain
$\cD_\cL$. Therefore, as in \cite{BR01} or \cite{BCS12}, it is necessary to make use of a slight modification of the usual strategy
of change of probability given in \eqref{DEFET}. There exists a unique $a_T$, which belongs to the interior of
$\cD_\cL$ and converges to its border $a_{c}$, solution of the implicit equation
\begin{equation}
\label{EQIMP1}
L^{\prime}(a_T) +\frac{1}{T} H^{\prime}(a_T)=0.
\end{equation}
It leads to the decomposition $\dP(\wh{\theta}_T \geq c)=A_TB_T$ with
\begin{eqnarray}
\label{DEFAT1} A_T &=& \exp (TL_T(a_T)),\\
\label{DEFBT1} B_T &=& \dE_T\left[\exp (-Z_T(a_T))\rI_{Z_T(a_T) \geq 0}\right],
\end{eqnarray}
where $\dE_T$ stands for the expectation after the time-varying change of probability
\begin{equation}
\label{DEFET1}
 \frac{\dd \dP_T}{\dd \dP} = \exp \left(Z_T(a_T) - TL_T(a_T)\right).
\end{equation}
We deduce from \eqref{E-DEFL2}, \eqref{E-DEFH2} together with \eqref{EQIMP1} that
\begin{equation*}
(\theta^2 T + \gamma^2) \tau(a_T)= \frac{\theta^2 ( c(2\varphi(a_T) -a_T)+\theta(\theta - 3c))}{\varphi(a_T)(\varphi(a_T)+c)}
\end{equation*}
where $\varphi(a)=\sqrt{\theta^2+ 2 a c}$ and $\tau(a)= \varphi(a) - a - \theta$. Consequently, it follows from
straightforward calculations that
\begin{eqnarray}
\label{CVGPHIT1}
& &\lim_{T \rightarrow \infty} T  (\varphi(a_T)+2c- \theta)= - \frac{c}{3c-\theta}, \\
\label{CVGAT1}
& &\lim_{T \rightarrow \infty} T  (a_T-a_c)= - \frac{(2c-\theta)}{3c-\theta}, \\
\label{CVGTAUT1}
& &\lim_{T \rightarrow \infty} T  \tau(a_T)= \frac{c-\theta}{3c-\theta}.
\end{eqnarray}
Moreover, we can show via \eqref{E-LOGDETMT} and \eqref{E-LASTLAP} that $R_T(a_T)=\cR_T(a_T,-ca_T)$ remains bounded when $T$
goes to infinity. Hence, Lemma \ref{L-LEMFOND2} together with \eqref{CVGPHIT1}, \eqref{CVGAT1}, and \eqref{CVGTAUT1} imply that
\begin{eqnarray}
\nonumber
A_T & = & \exp( TL(a_T)+H(a_T))\Bigl(1+o(1)\Bigr),\\
 & = & \exp\Bigl(  -TI_{\theta}(c)-\frac{\gamma^2}{\theta^2} (2c-\theta) \Bigr)
\left(\frac{2e T (2c-\theta)^3(3c-\theta)}{\theta^2(c-\theta)} \right)^{\!\!1/2} 
 \!\!\Bigl(1+o(1)\Bigr).
 \label{DEVAT1}
\end{eqnarray}
Moreover, the second term $B_T$ can be rewritten as 
\begin{equation}
\label{NBT1}
 B_T=\dE_T \Bigl[\exp (-a_T T U_T)\rI_{U_T \geq 0}\Bigr]
 \hspace{0.5cm} \text{where} \hspace{0.5cm} U_T=\frac{Z_T(1)}{T}.
\end{equation}

\begin{lem}
\label{L-BT1}
 For $c> \theta/3$ with $c\neq 0$, we have 
\begin{equation}
\label{DEVBT1}
 B_T=\frac{1}{a_c b_c T \sqrt{2 \pi e} }\Bigl(1+o(1)\Bigr)
\end{equation}
where
\begin{equation}
\label{DEFBC}
b_c=-L^\prime(a_c)= \frac{3c- \theta}{2(2c - \theta)}.
\end{equation}
\end{lem}

\begin{proof}
Denote by $\Phi_T$ the characteristic function of $U_T$ under $\mathbb{P}_T$. We infer from \eqref{DEFET1} that
for all $u \in \dR$,
\begin{equation}
\label{PHIT1}
\Phi_T(u)=\exp \left(T L_T\left(a_T+\frac{iu}{T}\right)-TL_T(a_T)\right).
\end{equation}
Moreover, we obtain from \eqref{CVGPHIT1} and \eqref{CVGAT1} that for $T$ large enough
and for all $u \in \dR$ such that $|u|=o(T)$,
\begin{equation*}
\exp \left(T L\left(a_T+\frac{iu}{T}\right)-TL(a_T)\right)=\exp \left(-ib_c u - \frac{\sigma_c^2 u^2}{2T} \right) 
\left(1+O\left( \frac{|u|^3}{T^2}\right)\right)
\end{equation*}
where $\sigma^2_{c}$ and $b_c$ are given by \eqref{DEFACSIGC2} and \eqref{DEFBC}.
Consequently, as soon as $|u|=o(T^{2/3})$,
\begin{equation}
\exp \left(T L\left(a_T+\frac{iu}{T}\right)-TL(a_T)\right)=\exp \left(-ib_c u - \frac{\sigma_c^2 u^2}{2T} \right) 
\Bigl(1+o(1)\Bigr)
\label{CVGCFL1}
\end{equation}
and the remainder $o(1)$ is uniform. By the same token,
\begin{equation}
\label{CVGCFH1}
\lim_{T\rightarrow \infty}
\exp \left(  H\Bigl(a_T+\frac{iu}{T}\Bigr)-H(a_T) \right)
=\frac{1}{\sqrt{1-2ib_{c}u}}.
\end{equation}
Therefore, we deduce from Lemma \ref{L-LEMFOND2} together with \eqref{CVGCFL1}, \eqref{CVGCFH1} and the
boundedness of $R_T(a_T)$,
that for all $u\in \dR$ such that $|u|=o(T^{2/3})$,
\begin{equation}
\label{CVGFINPHIT1}
 \Phi_T(u)  = \Phi(u) \exp \left( - \frac{\sigma_c^2 u^2}{2T} \right)
 \Bigl(1+o(1)\Bigr)
\end{equation}
where 
$$
\Phi(u)=\frac{1}{\sqrt{1-2ib_cu}}\exp\Bigl(-ib_cu \Bigr).
$$
It means that the distribution of $U_T$ under $\dP_T$ converges to $b_c(\xi^2-1)$, where 
$\xi$ stands for an $\cN(0,1)$ random variable. It also implies that, for $T$ large enough, 
$\Phi_T$ belongs to $L^2(\mathbb{R})$. Hereafter, we deduce from Parseval's formula that $B_T$, given by \eqref{NBT1}, can be rewritten as
\begin{equation}
 \label{NNBT1}
 B_T=\frac{1}{2\pi T a_T } \int_{\mathbb{R}}\left(1+\frac{iu}{a_T T}\right)^{-1}\Phi_T(u)\,du.
\end{equation}
We split $B_T$ into two terms, $B_T=C_T+D_T$, where
\begin{eqnarray}
\label{DEFBSC}
C_T&=&
\frac{1}{2\pi Ta_T} \int_{|u| \leq s_T}
\left(1+\frac{iu}{Ta_T}\right)^{-1}\Phi_{T}(u)\, du, \\
D_T&=&
\frac{1}{2\pi Ta_T}  \int_{|u| > s_T}
\left(1+\frac{iu}{ Ta_T}\right)^{-1}\Phi_{T}(u)\, du
\label{DEFBSD}
\end{eqnarray}
with $s_T=T^{2/3}$. On the one hand, it follows from \eqref{DEFBSD} that $D_T$ is negligible, as
\begin{equation}
\label{FINBSD}
D_T=o\left(\exp\left( -\frac{\sigma_c^2 s_T^2}{2T} \right) \right).
\end{equation}
On the other hand, we find from \eqref{CVGFINPHIT1} that for $T$ large enough
\begin{equation*}
2\pi T a_T C_T=\int_{|u| \leq s_T}
\Phi(u)\exp\left(-\frac{\sigma_c^2u^2}{2T}\right) \Bigl(1+o(1)\Bigr)
\, du,
\end{equation*}
which leads, thanks to Lemma 7.3 in \cite{BR01}, to
\begin{equation}
\label{FINBSC}
\lim_{T\rightarrow \infty} 2\pi T a_T C_T= \frac{\sqrt{2\pi}}{b_c \sqrt{e}}.
\end{equation}
Hence, \eqref{FINBSD} together with \eqref{FINBSC} clearly imply \eqref{DEVBT1}.
\end{proof}

\noindent
Finally, we immediately deduce \eqref{E-SLDPT3}  from \eqref{DEVAT1} and \eqref{DEVBT1}. 
\demend

\subsection{Proof of Theorem \ref{T-SLDPT}, third part}
Assume now that $c=\theta/3$ which means that $a_c=a_{\theta}$ with
$a_{\theta}=-4\theta/3$. There exists a unique $a_T$, which belongs to the interior of
$\cD_\cL$ and converges to its border $a_{\theta}$, solution of the implicit equation
\begin{equation}
\label{EQIMP2}
L^{\prime}(a_T) +\frac{1}{T} H^{\prime}(a_T)=0.
\end{equation}
We deduce from \eqref{E-DEFL2}, \eqref{E-DEFH2} together with \eqref{EQIMP2} that
\begin{equation*}
(\theta^2 T + \gamma^2) \tau(a_T)= \frac{\theta^2 c(2\varphi(a_T) -a_T)}{\varphi(a_T)(\varphi(a_T)+c)}
\end{equation*}
where $\varphi(a)=\sqrt{\theta^2+ 2 a c}$ and $\tau(a)= \varphi(a) - a - \theta$. We obviously have
$$
\tau(a)=-\frac{(\varphi(a)+c)(\varphi(a)-\theta)}{2c}
$$
which leads to 
\begin{equation*}
(\theta^2 T + \gamma^2) (\varphi(a_T)+c)^2= \frac{2\theta^2 c^2(a_T- 2\varphi(a_T))}{\varphi(a_T)(\varphi(a_T)- \theta)}.
\end{equation*}
It implies after some elementary calculations that
\begin{eqnarray}
\label{CVGPHIT2}
& &\lim_{T \rightarrow \infty} T  (\varphi(a_T)+c)^2= - \frac{\theta}{3}, \\
\label{CVGAT2}
& &\lim_{T \rightarrow \infty} T  (a_T-a_\theta)^2= - \frac{\theta}{3}, \\
\label{CVGTAUT2}
& &\lim_{T \rightarrow \infty} \sqrt{T}  \tau(a_T)= 2 \sqrt{ -\frac{\theta}{3}}.
\end{eqnarray}
Hereafter, we shall make use of the decomposition $\dP(\wh{\theta}_T \geq c)=A_TB_T$ given by
\begin{eqnarray}
\label{DEFAT2} A_T &=& \exp (TL_T(a_T)),\\
\label{DEFBT2} B_T &=& \dE_T\left[\exp (-Z_T(a_T))\rI_{Z_T(a_T) \geq 0}\right],
\end{eqnarray}
where $\dE_T$ stands for the expectation after the time-varying change of probability
\begin{equation}
\label{DEFET2}
 \frac{\dd \dP_T}{\dd \dP} = \exp \left( Z_T(a_T) - TL_T(a_T)\right).
\end{equation}
We obtain from \eqref{E-LOGDETMT} and \eqref{E-LASTLAP} that $R_T(a_T)=\cR_T(a_T,-ca_T)$ remains bounded when $T$
goes to infinity. Hence, it follows from Lemma \ref{L-LEMFOND2} together with \eqref{CVGPHIT2}, \eqref{CVGAT2}, and \eqref{CVGTAUT2} that
\begin{equation}
\label{DEVAT2}
 A_T = \frac{\sqrt{2}}{3}\exp\Bigl(  -TI_{\theta}(c)+ \frac{\gamma^2}{3\theta} \Bigr)
\left(\frac{-\theta e T}{3}\right)^{\!\!1/4} 
 \!\!\Bigl(1+o(1)\Bigr).
\end{equation}
On the other hand, $B_T$ can be rewritten as 
\begin{equation}
\label{NBT2}
 B_T=\dE_T \left[\exp (-a_T\sqrt{T} U_T)\rI_{U_T \geq 0}\right]
 \hspace{0.5cm} \text{where} \hspace{0.5cm}
 U_T=\frac{Z_T(1)}{\sqrt{T}}.
\end{equation}

\begin{lem}
\label{L-BT2}
 For $ c =\theta/3$, we have 
\begin{equation}
\label{DEVBT2}
 B_T=\frac{1}{4 \pi \sqrt{a_\theta T} }\exp(-1/4)\Gamma(1/4)\Bigl(1+o(1)\Bigr).
\end{equation}
\end{lem}

\begin{proof}
Via the same lines as in the proof of Lemma \ref{L-BT}, we find that the characteristic function $\Phi_T$ of $U_T$, under $\dP_T$,
belongs to $L^2(\dR)$. Hence, it follows from Parseval's formula that 
\begin{equation}
 \label{NNBT2}
 B_T=\frac{1}{2\pi a_T\sqrt{T}} \int_{\mathbb{R}}\left(1+\frac{iu}{a_T\sqrt{T}}\right)^{-1}\Phi_T(u)du.
\end{equation}
However, we obtain from \eqref{CVGPHIT2} and \eqref{CVGAT2} that
\begin{equation}
\label{CVGCFL2}
\lim_{T\rightarrow \infty}
T\Bigl(L\Bigl(a_T+\frac{iu}{\sqrt{T}}\Bigr)-L(a_T) \Bigr)=-id_{\theta}u -\frac{\sigma^2_{\theta}u^2}{2}
\end{equation}
where $\sigma^2_{\theta}$ is given by \eqref{DEFACSIGC4} and $d_{\theta}= \sigma_{\theta}/\sqrt{2}$.
By the same token,
\begin{equation}
\label{CVGCFH2}
\lim_{T\rightarrow \infty}
\exp \left(  H\Bigl(a_T+\frac{iu}{\sqrt{T}}\Bigr)-H(a_T) \right)
=\frac{1}{\sqrt{1-2id_{\theta}u}}.
\end{equation}
Therefore, we deduce from Lemma \ref{L-LEMFOND2} together with \eqref{CVGCFL2}, \eqref{CVGCFH2} and the boundedness
of $R_T(a_T)$, 
the pointwise convergence 
\begin{equation}
\label{CVGFINPHIT2}
\lim_{T \rightarrow \infty}  \Phi_T(u)  = \Phi(u)= \frac{1}{\sqrt{1-2id_{\theta}u}}\exp\Bigl(-id_{\theta}u -\frac{\sigma^2_{\theta}u^2}{2} \Bigr).
\end{equation}
It shows that the distribution of $U_T$ under $\dP_T$ converges to $\sigma_{\theta}\zeta + d_{\theta}(\xi^2-1)$, where $\zeta$ and
$\xi$ are two independent random variables sharing the same $\cN(0,1)$ distribution. Finally, 
we obtain from \eqref{NNBT2} together with \eqref{CVGFINPHIT2} and the
Lebesgue dominated convergence theorem that
\begin{equation*}
 B_T=\frac{1}{2 \pi a_T \sqrt{T}}\ \int_\dR \Phi(u)\, du \Bigl(1+o(1)\Bigr)
 =\frac{1}{4 \pi \sqrt{a_\theta T} }\exp(-1/4)\Gamma(1/4)\Bigl(1+o(1)\Bigr)
\end{equation*}
which achieves the proof of Lemma \ref{L-BT2}.
\end{proof}

\noindent
The proof of \eqref{E-SLDPT4} immediately follows from the conjunction of \eqref{DEVAT2} and \eqref{DEVBT2}. 
\demend

\subsection{Proof of Theorem \ref{T-SLDPT}, fourth part}
Assume now that $c=0$. We want to obtain the leading asymptotic behavior of 
$\dP(\wh{\theta}_T \geq 0)=\dP(X_T^2-2X_T\overline{X}_T \geq T)$.
For all $\alpha>0$, we have the decomposition $\mathbb{P}(X_T^2-2X_T\overline{X}_T \geq T)=A_T+B_T$ where
\begin{eqnarray*}
A_T &=& \mathbb{P}(X_T^2-2X_T\overline{X}_T \geq T, \abs{\overline{X}_T} \leq \alpha), \\
B_T &=& \mathbb{P}(X_T^2-2X_T\overline{X}_T \geq T, \abs{\overline{X}_T} > \alpha).
\end{eqnarray*}
First of all, if 
$$\alpha = \abs{\frac{\gamma}{\theta}}+\frac{2}{\sqrt{-\theta}},$$
it is not hard to see that $B_T$ is negligible. As a matter of fact, we deduce from
the simple upper bound $B_T \leq \mathbb{P}(\abs{\overline{X}_T} > \alpha)$
together with \eqref{ineg1} that
\begin{equation*}
\limsup_{T \rightarrow + \infty} \frac{1}{T}\log B_T \leq 
2 \theta, \hspace{1cm}
B_T = o\left(\exp\left(\frac{3 \theta T}{2} \right)\right).
\end{equation*}
Next, we recall that the sequence $(\wh{\theta}_T)$ satisfies an LDP with good rate function $I_{\theta}$ given by \eqref{RATET}.
Consequently,
$$\lim_{T \rightarrow + \infty} \frac{1}{T} \log \dP(\wh{\theta}_T \geq 0)=
\lim_{T \rightarrow + \infty} \frac{1}{T} \log \dP(X_T^2 -2 X_T \overline{X}_T \geq T) = \theta$$
which clearly implies that
$\dP(X_T^2 -2 X_T \overline{X}_T \geq T)=A_T(1+o(1))$.
From now on, it only remains to establish the leading asymptotic behavior of $A_T$. 
We already saw at the beginning of Section \ref{S-KSL} that the random vector
$$
\begin{pmatrix}
X_T \\
\overline{X}_T
\end{pmatrix} 
\sim \cN \left(
\begin{pmatrix}
m_T\\
\mu_T
\end{pmatrix}, \Gamma_T(\theta) \right).
$$
Therefore,
$$A_T = \int_{-\alpha}^{\alpha} h_T(x) f_T(x)\,dx
\hspace{0.5cm} \text{where} \hspace{0.5cm}
h_T(x)=\dP(X_T^2 -2 X_T \overline{X}_T \geq T | \overline{X}_T=x)
$$
and $f_T$ is the Gaussian probability density function of $\overline{X}_T$. 
Moreover, as $c_T >0$, the conditional distribution of $X_T$ given $\overline{X}_T=x$ is $\cN(\nu_T,s^2_T)$ with
$\nu_T = m_T + b_T(x-\mu_T)/cT$ and  $s_T^2 = a_T-b_T^2/c_T$.
Furthermore, for all $x \in \dR$, $h_T(x)$ can be rewritten as
$$h_T(x)=\dP \left(\frac{X_T-\nu_T}{s_T} \leq -y_T \Big| \overline{X}_T=x\right)+\dP \left(\frac{X_T-\nu_T}{s_T} \geq z_T \Big| 
\overline{X}_T=x\right)$$
where
$$y_T = \frac{-x +\sqrt{x^2+T}+\nu_T}{s_T}
\hspace{1cm} \textrm{and} \hspace{1cm}
z_T =\frac{x +\sqrt{c^2+T}-\nu_T}{s_T}.
$$
One can easily check that
$$
\liminf_{T \rightarrow + \infty} \Bigl\{y_T,z_T\Bigr\} \geq 
\liminf_{T \rightarrow + \infty} 
\frac{-\alpha+\sqrt{T}-m_T-b_T c_T^{-1}( \alpha-\mu_T)}{s_T}= +\infty.$$
It follows from standard asymptotic analysis of Gaussian distribution tails that
$$h_T(x)=\frac{1}{y_T\sqrt{2\pi}}\exp\left(-\frac{y_T^2}{2} \right)\Bigl(1+o(1)\Bigr) + \frac{1}{z_T\sqrt{2\pi}}
\exp\left(-\frac{z_T^2}{2} \right)\Bigl(1+o(1)\Bigr)$$
where $o(1)$ is uniform with respect to $x$. We split $A_T$ into two terms, $A_T = C_T+D_T$ 
\begin{eqnarray*}
C_T &=& \int_{-\alpha}^{\alpha} \frac{1}{y_T\sqrt{2\pi} } \exp\left(-\frac{y_T^2}{2} \right)
\frac{1}{\sqrt{2\pi c_T}} \exp\left(-\frac{(x-\mu_T)^2}{2c_T} \right) dx \ \Bigl(1+o(1)\Bigr),
\\
D_T &=& \int_{-\alpha}^{\alpha} \frac{1}{z_T\sqrt{2\pi} } \exp\left(-\frac{z_T^2}{2} \right)
\frac{1}{\sqrt{2\pi c_T}} \exp\left(-\frac{(x-\mu_T)^2}{2c_T} \right) dx \ \Bigl(1+o(1)\Bigr).
\end{eqnarray*}
We find from a careful asymptotic expansion inside the integral $C_T$ together with the change of variables
$y=-\theta(x+\gamma/\theta)\sqrt{T}$ and Lebesgue's dominated convergence theorem, that
\begin{equation*}
\lim_{T \rightarrow +\infty} \sqrt{2\pi}\sqrt{-2\theta T}e^{-\theta T}C_T = 
\frac{1}{\sqrt{2\pi}}
\int_{-\infty}^{+\infty} \exp\left( 2y+\frac{\gamma^2}{\theta} -\frac{y^2}{2}\right) dy
=\exp\left( \frac{\gamma^2}{\theta}+2 \right).
\end{equation*}
By the same token, we also obtain that
$$\lim_{T \rightarrow +\infty} \sqrt{2\pi}\sqrt{-2\theta T}e^{-\theta T}D_T = \exp\left( \frac{\gamma^2}{\theta}+2 \right)$$
which is exactly what we wanted to prove.
\demend

\section*{Appendix A: Proof of Lemma \ref{L-LEMFOND1}.}
\renewcommand{\thesection}{\Alph{section}}
\renewcommand{\theequation}{\thesection.\arabic{equation}}
\setcounter{section}{1}
\setcounter{equation}{0}

For all $(a,b,c) \in \dR^3$, let
$$
\cZ_T(a,b,c)= a\sqrt{T} X_T + b \int_0^T X_t^2\, dt + c  \int_0^T X_t \,dt.
$$
We shall calculate the limit $\Lambda$ of the normalized cumulant generating function $\Lambda_T$ of the random
variable $\cZ_T(a,b,c)$.
First of all, as in \cite{DY93}, it follows from Girsanov's formula associated with \eqref{OUPS} that
\begin{eqnarray*}
\Lambda_T(a,b,c) &\!=\!& \frac{1}{T} \log \dE\Bigl[\exp(\cZ_T(a,b,c))\Bigr], \\
&\!=\!& \frac{1}{T} \log \dE_{\vp, \delta} \!\left[ \exp\Bigl((\theta - \vp) \!\int_0^T \!\! X_t dX_t +\frac{1}{2}( 2 b  - \theta^2 + \vp^2) \!\int_0^T\!\! X_t^2 \,dt 
- \xi_T\Bigr) \right]
\end{eqnarray*}
where $\dE_{\vp,\delta}$ stands for the expectation after the change of probability, 
$$
\frac{\dd \dP_{\vp,\delta}}{\dd \dP_{\theta, \gamma}}= 
\exp \left( ( \vp- \theta) \int_0^T X_t dX_t - \frac{1}{2} (\vp^2 - \theta^2) \int_0^T X_t^2\,dt + \zeta_T
\right)
$$
with $\xi_T =  \zeta_T -a \sqrt{T} X_T -cT\overline{X}_T$ and
$$
\zeta_T = (\theta \gamma -\vp \delta) T\overline{X}_T - (\gamma-\delta) X_T + \frac{1}{2}(\gamma^2 - \delta^2)T. 
$$
Consequently, if we assume that $\theta^2-2b>0$ and if we choose $\vp=\sqrt{\theta^2-2b}$ and $\delta=0$, 
$\Lambda_T(a,b,c)$ reduces to
\begin{equation}
\label{DECMGFLB1}
\Lambda_T(a,b,c) = \frac{\vp - \theta - \gamma^2}{2}+ \frac{1}{T} \log \dE_{\vp,0} \left[ \exp\Bigl(-\frac{1}{2} V_T^\prime J V_T
+ U_T^\prime V_T \Bigr) \right]
\end{equation}
where the vectors $U_T$ and $V_T$ are given by
$$
U_T=\begin{pmatrix}
a \sqrt{T} + \gamma \\
T(c-\theta \gamma)
\end{pmatrix}
\hspace{1cm} \text{and} \hspace{1cm}
V_T=\begin{pmatrix}
X_T \\
\overline{X}_T
\end{pmatrix}
$$
and $J$ is the diagonal matrix of order two
$$
J=\begin{pmatrix}
\, \vp - \theta \, & 0 \\
0 & 0
\end{pmatrix}.
$$
Under the new probability $\dP_{\vp,0}$, $V_T$ is Gaussian random vector with zero mean and covariance matrix $\Gamma_T(\vp)$ given by 
\eqref{DEFGAMT}. Denote by $M_T(a,b,c)$ the square matrix of order two
$$
M_T(a,b,c)= \rI_2 +J \Gamma_T(\vp)=\begin{pmatrix}
1+(\vp - \theta) a_T(\vp) & (\vp - \theta) b_T(\vp)  \\
0 & 1
\end{pmatrix}
$$
where $\rI_2$ stands for the identity matrix of order two.
We clearly have
$$\det M_T(a,b,c) = 1+(\vp- \theta) a_T(\vp)=1+\frac{\vp- \theta}{2\vp}\Bigl(e^{2\vp T}-1\Bigr)$$
which leads to
\begin{equation}
\label{LIMDETMT}
\lim_{T \rightarrow \infty} \frac{\det M_T(a,b,c)}{e^{2\vp T}}=\frac{\vp- \theta}{2\vp}.
\end{equation}
Hence, as $\theta<0<\vp$, it follows from \eqref{LIMDETMT} that for $T$ large enough, the matrix $M_T(a,b,c)$
is positive definite. It is also not hard to see from \eqref{DEFGAMT} that
\begin{equation}
\label{LIMDETGAMT}
\lim_{T \rightarrow \infty} \frac{T \det \Gamma_T(\vp)}{e^{2\vp T}}=\frac{1}{2\vp^3}.
\end{equation}
Therefore, we obtain from standard Gaussian calculations that
\begin{equation}
\label{DECMGFLB2}
\Lambda_T(a,b,c) = \frac{\vp-\theta - \gamma^2}{2}- \frac{1}{2T} \log \Bigl(\det M_T(a,b,c) \Bigr)
+ \frac{1}{T} H_T(a,b,c)
\end{equation}
where
$H_T(a,b,c)= 2^{-1} U_T^\prime \Gamma_T(\vp) M_T^{-1}(a,b,c)U_T$.
On the one hand, we immediately obtain from \eqref{LIMDETMT} that
\begin{equation}
\label{LIMLOGDET}
\lim_{T \rightarrow \infty} \frac{1}{2T} \log \Bigl(\det M_T(a,b,c) \Bigr)=\vp.
\end{equation}
On the other hand, we clearly have
$$
H_T(a,b,c)= \frac{1}{2\det M_T(a,b,c)} \Bigl( (a\sqrt{T} + \gamma)^2 a_T(\vp) +2T d_T(\vp)
+T^2(c-\theta \gamma)^2 e_T(\vp)\Bigr)
$$
where $d_T(\vp)=(a\sqrt{T} + \gamma)(c-\theta \gamma)b_T(\vp)$ and $e_T(\vp)=c_T(\vp)+ (\vp - \theta) \det \Gamma_T(\vp)$.
Consequently, we obtain from \eqref{LIMDETMT} and \eqref{LIMDETGAMT} that
\begin{equation}
\label{LIMHT}
\lim_{T \rightarrow \infty} \frac{1}{T} H_T(a,b,c)  = \frac{1}{2}\left(\frac{a^2}{\vp - \theta} + \frac{(c-\theta \gamma)^2}{\vp^2} \right).
\end{equation}
Finally, we deduce from \eqref{DECMGFLB2} together with \eqref{LIMLOGDET} and \eqref{LIMHT} that
\begin{equation}
\label{LIMLBT}
\lim_{T \rightarrow \infty} \Lambda_T(a,b,c) = -\frac{1}{2}\left(\theta  + \vp+ \gamma^2\right)
+\frac{1}{2}\left(\frac{a^2}{\vp- \theta}\right)+\frac{1}{2} \left( \frac{c-\theta \gamma}{\vp}\right)^2,
\end{equation}
which is exactly what we wanted to prove.
\demend

\section*{Appendix B: Proof of Lemma \ref{L-LEMFOND2}.}
\renewcommand{\thesection}{\Alph{section}}
\renewcommand{\theequation}{\thesection.\arabic{equation}}
\setcounter{section}{2}
\setcounter{equation}{0}

Our goal is to establish the full asymptotic expansion for the normalized cumulant generating function
$\cL_T$ of the random variable $\cZ_T(a,b)$. First of all, as in the proof of Lemma \ref{L-LEMFOND1}, it follows from Girsanov's formula 
associated with \eqref{OUPS} that
\begin{eqnarray*}
\cL_T(a,b) &\!=\!& \frac{1}{T} \log \dE\Bigl[\exp(\cZ_T(a,b))\Bigr], \\
&\!=\!& \frac{1}{T} \log \dE_{\vp, \delta}\! \left[ \exp\Bigl((a+\theta - \vp) \!\int_0^T \!\! X_t dX_t 
\!+\!\frac{1}{2}( 2 b  - \theta^2 + \vp^2) \!\int_0^T\!\! X_t^2 \,dt - \xi_T\Bigr) \right]
\end{eqnarray*}
where $\xi_T = aX_T \overline{X}_T+ \zeta_T  +b T (\overline{X}_T)^2$ and
$ \zeta_T = (\theta \gamma -\vp \delta) T\overline{X}_T - (\gamma-\delta) X_T + (\gamma^2 - \delta^2)T/2$.
Consequently, if we assume that $\theta^2-2b>0$ and if we choose $\vp=\sqrt{\theta^2-2b}$ and $\delta=0$, $\cL_T(a,b)$ reduces to
\begin{equation*}
\label{DECOMGF1}
\cL_T(a,b) =  \frac{\tau- \gamma^2}{2}+ \frac{1}{T} \log \dE_{\vp,0} \left[ \exp\Bigl(-\frac{1}{2} V_T^\prime J_T V_T
+ \gamma U_T^\prime V_T \Bigr) \right]
\end{equation*}
with $\tau=\varphi -(a+\theta)$, where the vectors $U_T$ and $V_T$ are given by
$$
U_T=\begin{pmatrix}
\, 1 \\
-\theta T
\end{pmatrix}
\hspace{1cm} \text{and} \hspace{1cm}
V_T=\begin{pmatrix}
X_T \\
\overline{X}_T
\end{pmatrix}
$$
and $J_T$ is the diagonal matrix of order two
$$
J_T=\begin{pmatrix}
\, \tau \, & a \\
a & 2bT
\end{pmatrix}.
$$
We already saw in Appendix A that under the new probability $\dP_{\vp,0}$, $V_T$ is Gaussian random vector with zero mean and covariance matrix 
$\Gamma_T(\vp)$ given by \eqref{DEFGAMT}. Let $M_T(a,b)$ be the square matrix of order two
$$
M_T(a,b)= \rI_2 +J_T \Gamma_T(\vp)=\begin{pmatrix}
1+\tau a_T(\vp) +a  b_T(\vp)  & \tau b_T(\vp)+a  c_T(\vp) \vspace{1ex} \\
 2bT b_T(\vp) +a  a_T(\vp) & 1+2bTc_T(\vp)+a  b_T(\vp) 
\end{pmatrix}.
$$
It is not hard to see that
\begin{equation}
\label{DETMTAB}
\det(M_T(a,b))=1+2ab_T(\vp)+2bTc_T(\vp)+\tau a_T(\vp)+(2\tau T b - a^2) \det \Gamma_T(\vp).
\end{equation}
Hence, we deduce from \eqref{LIMDETGAMT} that
\begin{equation}
\label{LIMDETMTN}
\lim_{T \rightarrow \infty} \frac{\det M_T(a,b)}{e^{2\vp T}}=\frac{\tau \theta^2}{2\vp^3}.
\end{equation}
Consequently, as soon as $\tau >0$, we find from \eqref{LIMDETMTN} that for $T$ large enough, the matrix $M_T(a,b)$
is positive definite. Therefore, it follows from standard Gaussian calculations that
\begin{equation}
\label{DECOMGF}
\cL_T(a,b)= \frac{\tau - \gamma^2}{2}- \frac{1}{2T} \log \Bigl(\det M_T(a,b) \Bigr)
+ \frac{\gamma^2}{2T} U_T^\prime \Gamma_T(\vp) M_T^{-1}(a,b)U_T.
\end{equation}
We shall now to improve convergence \eqref{LIMDETMTN} as follows.
We have for $T$ large enough
\begin{equation}
\label{E-LOGDETMT}
\frac{1}{2T} \log \Bigl(\det M_T(a,b) \Bigr) = \vp+\frac{1}{2T} \log \left( \frac{\tau \theta^2}{2\vp^3} \right) + \frac{1}{2T} \log 
\Bigl(1+\cK_T(a,b)\Bigr) +r_T(a,b)
\end{equation} 
where $\cK_T(a,b)$ is given by
\begin{equation*}
 \label{DEFKT}
 \cK_T(a,b)=\frac{1}{\tau \varphi \theta^2}\left( \frac{1}{T} ( 2b \varphi +2a \varphi^2-a^2\varphi - 4\tau b ) +\frac{2a^2}{T^2}\right),
\end{equation*}
and $r_T(a,b)$ is such that $\abs{r_T(a,b)} \leq P(a,b, \vp, T)e^{-\vp T}$
with $P$ a rational function.
As a matter of fact, we already saw from \eqref{DETMTAB} that
$$\det(M_T(a,b))=1+\Delta_T(a,b)+(2\tau T b - a^2) \det \Gamma_T(\vp)$$
where $\Delta_T(a,b)=2ab_T(\vp)+2bT(\vp)c_T(\vp)+\tau a_T(\vp)$.
We obtain from straightforward calculations that
$$\det \Gamma_T(\vp) = \frac{(T\vp-2)e^{2\vp T} +4e^{\vp T}-(T\vp +2)}{2\vp^4 T^2}.$$
In addition, we deduce from the definition of $a_T(\vp), b_T(\vp)$ and $c_T(\vp)$ that
$$
\Delta_T(a,b) = \left( \frac{\tau}{2\vp}+\frac{a\vp+b}{\vp^3 T} \right) e^{2\vp T} -\frac{2a\vp+4b}{\vp^3 T} e^{\vp T}
+\frac{2a\vp+6b+4bT\vp-\tau\vp^2 T}{2\vp^3 T}.
$$
Consequently,
\begin{eqnarray*}
\det(M_T(a,b))&=&\left( \frac{\tau(\vp^2+2b)}{2\vp^3}+\frac{2b\vp+2a\vp^2-4\tau b-a^2 \vp}{2\vp^4 T} +\frac{2a^2}{2\vp^4 T^2} \right) e^{2\vp T}\\
&& \ +P_1(a,b,\varphi,T)e^{\varphi T}+ P_2(a,b,\varphi,T)
\end{eqnarray*}
where $P_1$ and $P_2$ are rational functions, leading to \eqref{E-LOGDETMT}.
Furthermore, concerning the last term in \eqref{DECOMGF}, we have for $T$ large enough
\begin{equation}
\label{E-LASTLAP}
\frac{\gamma^2}{2T} U_T^\prime \Gamma_T(\vp) M_T^{-1}(a,b)U_T= \frac{\gamma^2}{2}-\frac{1}{2T}\frac{\gamma^2(\vp+a+\theta)}{\theta^2}+\frac{1}{T}\cJ_T(a,b)+\rho_T(a,b),
\end{equation}
where $\cJ_T(a,b)$ is given by
$$
 \cJ_T(a,b)=T\frac{\gamma^2}{2}\Bigl(
\frac{I_T(f,g)}{I_T(h,\ell)} -I_T(k,0) \Bigr) 
$$
with $I_T(x,y)= {\displaystyle 1+\frac{x}{T}+\frac{y}{T^2}},$
$$
f= \frac{2\theta a \vp +2a\theta^2+4\theta b}{\theta^2\vp\tau}, \quad 
g=-\frac{4(b+\theta a)}{\theta^2\vp\tau}, \quad h=\frac{2a\theta^2+4b\theta-a^2\vp-2b\vp}{\theta^2\vp\tau},
$$
$$
k=-\frac{\vp+a+\theta}{\theta^2}, \quad 
\ell=\frac{2a^2}{\theta^2\vp\tau}.
$$
Moreover, the remainder $\rho_T(a,b)$ is such that $| \rho_T(a,b) | \leq Q(a,b, \vp, T) e^{-\vp T}$ with $Q$ a rational function.
As a matter of fact, we have previously remark that for $T$ large enough, $\det(M_T(a,b))>0$. 
Consequently, $M_T^{-1}(a,b)$ is well defined. With 
very tedious but straightforward calculations, we obtain that
$$\frac{\gamma^2}{2T} U_T^\prime \Gamma_T(\vp) M_T^{-1}(a,b)U_T=\frac{\gamma^2}{2}\frac{N_T(a,b)}{D_T(a,b)}$$
with
\begin{eqnarray*}
N_T(a,b) &=& \left(\theta^2\vp\tau T^2+(2\theta a \vp +2a\theta^2+4\theta b)T-4(b+\theta a) \right) e^{2\vp T}\\
& & +\left(-2\theta (\theta a+2b)T+8(b+\theta a)\right) e^{\vp T} + Q_1(a,b, \vp, T)\\
D_T(a,b)&=& \left(\theta^2\vp\tau T^2+(2a\theta^2+4b\theta-a^2\vp-2b\vp)T+ 2a^2 \right) e^{2\vp T}\\
& & + \left(-4\theta (\theta a+2b)T-4a^2\right) e^{\vp T} + Q_2(a,b, \vp, T)\\
\end{eqnarray*}
where $Q_1$ and $Q_2$ are polynomial functions, which clearly implies
\eqref{E-LASTLAP}. Finally, Lemma \ref{L-LEMFOND2} follows from the conjunction
of \eqref{DECOMGF}, \eqref{E-LOGDETMT} and \eqref{E-LASTLAP}. \demend

\section*{Appendix C: Proof of Lemma \ref{L-LEMDEC}.}
\renewcommand{\thesection}{\Alph{section}}
\renewcommand{\theequation}{\thesection.\arabic{equation}}
\setcounter{section}{3}
\setcounter{equation}{0}

It follows from \eqref{DEFZTAB} together with  It\^o's formula \eqref{ITO} that
\begin{equation}
\label{NEWDECZTAB}
 \cZ_T(a,b) = \frac{a}{2} X_T^2 -\frac{aT}{2}-a\overline{X}_T X_T + b \int_0^T X_t^2 dt - b T (\overline{X}_T)^2.
\end{equation}
In addition, we already saw at the beginning of Section \ref{S-KSL} that we can split
$X_T=Y_T+m_T$ and $\overline{X}_T=\overline{Y}_T+ \mu_T$.
Consequently, we have the decomposition
\begin{equation*}
\cZ_T(a,b)= \cZ_T^0 (a,b)+\cZ_T^1 (a,b) + \cZ_T^2 (a,b),
\end{equation*}
where $\cZ_T^0 (a,b) = \dE\left[ \cZ_T(a,b)\right]$,
\begin{eqnarray*}
\cZ_T^1 (a,b) &=& a(m_T - \mu_T) Y_T -am_T \overline{Y}_T +2b \int_0^T (m_t-\mu_T) Y_t dt,\\
\cZ_T^2 (a,b) &=& \frac{a}{2}\Bigl(Y_T^2-\dE [Y_T^2]\Bigr) - a\Bigl(\overline{Y}_T Y_T - \dE[\overline{Y}_T Y_T]\Bigr) 
-bT\Bigl((\overline{Y}_T)^2-\dE\Bigl[(\overline{Y}_T)^2\Bigr]\Bigr) \\
& & \hspace{2cm} +b \left( \int_0^T Y_t^2 dt -\dE \left[ \int_0^T Y_t^2 dt\right]\right). 
\end{eqnarray*}
By using the same notations as in Chapters 2 and 6 of Janson \cite{Janson97}, we clearly have
\begin{equation*}
 \cZ_T^0 (a,b) \in H^{:0:}, \quad  \cZ_T^1 (a,b) \in H^{:1:}, \quad \cZ_T^2 (a,b) \in H^{:2:},
\end{equation*}
where $ H^{:n:}$ stands for the homogeneous chaos of order $n$. Hence, we deduce from Theorem 6.2 of \cite{Janson97} that
\begin{equation}
\label{DECKL}
\cZ_T(a,b) = \dE[\cZ_T(a,b)] + \sum_{k=1}^{\infty} \alpha_k ^{T}\left(\varepsilon_k^2-1\right) +\beta_k ^{T}\varepsilon_k 
\end{equation}
where $(\varepsilon_k)$ are independent standard $\cN(0,1)$ random variables. We also obtain from
Theorem 6.2 of \cite{Janson97}  that
\begin{equation}
\label{UPBSBETA}
\sum_{k=1}^{+\infty} (\beta_k^T)^2=\dE\left[\left(\cZ_T^1 (a,b)\right)^2\right].
\end{equation}
In addition, some rough estimates give us that the right-hand side of \eqref{UPBSBETA} is uniformly bounded by some
constant $B>0$, depending only on $a$ and $b$. As a matter of fact, it exists some constant $\zeta(a,b)>0$ such that
\begin{equation}
\label{UPBSBETAB}
 \dE\left[(\cZ_T^1 (a,b))^2\right] \leq \zeta(a,b)\Bigl(\dE[ Y_T^2 ]+ \dE[ (\overline{Y}_T)^2 ]+
 \dE[ \Delta_T^2]\Bigr),
\end{equation}
where
$$
 \dE\left[ Y_T^2 \right] = \int_0^T e^{2\theta (T-t)}dt \leq -\frac{1}{2\theta}, \hspace{0.5cm}
 \dE\left[ (\overline{Y}_T)^2 \right] \leq \frac{1}{T}\int_0^T \mathbb{E}\left[ Y_t^2 \right]ds \leq -\frac{1}{2\theta},
$$
and
\begin{eqnarray*}
\dE[ \Delta_T^2] &=& \dE\left[ \left(\int_0^T(m_t-\mu_T)Y_tdt\right)^2 \right] = \frac{\gamma^2}{\theta^2}
 \dE\left[ \left(\int_0^T\Bigl(\frac{1-e^{\theta T}}{\theta T}+e^{\theta t}\Bigr)Y_tdt\right)^2 \right],\\
 &\leq & \frac{2\gamma^2}{\theta^4 }\frac{1}{T} \int_0^T \dE\left[ Y_t^2 \right]dt + \frac{2\gamma^2}{\theta^2} \int_0^T e^{\theta t} dt\int_0^T e^{\theta t}
 \dE\left[ Y_t^2 \right]dt \leq \frac{-2\gamma^2}{\theta^5 }.
\end{eqnarray*}
Therefore, we obtain \eqref{BOUNDBETAT} from \eqref{UPBSBETA} and \eqref{UPBSBETAB}. It now remains to show that
it exists some constant $A>0$ that do not depend on $T$, such that $|\alpha_k^T| \leq A$ 
for all $k \geq 1$. Since $\cD_\cL$ is an open set and 
the origin belongs to the interior of $\cD_\cL$, it exists $\varepsilon >0$ such that 
$\{(xa,xb) \in \cD_\cL \ \slash   \abs{x} < \varepsilon \} \subset \cD_\cL$.
For all $(a,b) \in \cD_\cL$ and for $T$ large enough, we deduce from Lemma \ref{L-LEMFOND2} that
$\exp(T \cL_T(xa,xb)) =\dE[ \exp (x\cZ_T(a,b)) ]$ is finite.
It means that the Laplace transform of $\cZ_T(a,b)$ is well defined on $[-\varepsilon,\varepsilon]$. Hence, Theorem 6.2 of \cite{Janson97} 
ensures that the characteristic function of $\cZ_T(a,b)$ is analytic in the strip 
$$\Bigl\{ z \in \dC \ \slash \ \abs{\textrm{Im} z}<\frac{1}{2}(\max_{k \geq 1} \abs{\alpha_k^T})^{-1} \Bigr\}. $$
So, we necessarily obtain that for $T$ large enough,
$\max \abs{\alpha_k^T} < A$ 
with $A=1/2\veps$. Hereafter, the decomposition of $\cL_T(xa,xb)$ given in Lemma  \ref{L-LEMDEC}, directly follows
from equation (6.7) in Theorem 6.2 of Janson \cite{Janson97}. Our goal is now to pass through the limit in $\cL_T(xa,xb)$. 
If we choose $x\in \dR$ such that $|x|\leq 1/4A$, we have
\begin{equation}
\label{LIMNEGTAB}
\lim_{T \rightarrow \infty}\frac{1}{2T} \sum_{k=1}^{\infty}\frac{(x\beta_k ^{T})^2}{1-2x\alpha_k ^{T}} \leq 
\lim_{T \rightarrow \infty} \frac{B}{16A^2T}  = 0.
\end{equation}
Moreover, we deduce from Lemma \ref{L-LEMFOND2} that
\begin{equation}
\label{LIMLTAB}
\lim_{T \rightarrow \infty} \cL_T(xa,xb) = \cL(xa,xb)= -\frac{1}{2} \left(xa + \theta + \sqrt{\theta^2-2xb} \right).
\end{equation}
Furthermore, it follows from the properties of $(X_T)$ and $(\overline{X}_T)$ given at the beginning of Section \ref{S-KSL} that
$$
\lim_{T \rightarrow \infty} \dE[X_T^2]=\frac{\gamma^2}{\theta^2}- \frac{1}{2 \theta}
\hspace{1cm}  \text{and} \hspace{1cm} 
\lim_{T \rightarrow \infty} \dE[(\overline{X}_T)^2]=\frac{\gamma^2}{\theta^2},
$$
which clearly implies
$$
\lim_{T \rightarrow \infty}\frac{1}{T} \dE[X_T^2]=0, \hspace{0.5cm}\lim_{T \rightarrow \infty} \frac{1}{T} \int_0^T \dE[X_t^2]dt=\frac{\gamma^2}{\theta^2}- \frac{1}{2 \theta},
 \hspace{0.5cm}\lim_{T \rightarrow \infty} \dE[X_T\overline{X}_T]=0.
$$
Then, we find from \eqref{NEWDECZTAB} that
\begin{equation}
\label{LIMZTAB}
\lim_{T \rightarrow \infty} \frac{1}{T}  \dE[\cZ_T(xa,xb)]=-\frac{x}{2}\Bigl(a+ \frac{b}{\theta} \Bigr).
\end{equation}
Finally, we obtain from the decomposition of $\cL_T(xa,xb)$,
\eqref{LIMNEGTAB}, \eqref{LIMLTAB}, \eqref{LIMZTAB} that
\begin{equation}
\lim_{T \rightarrow +\infty} \frac{1}{2T} \sum_{k=1}^{\infty}\Bigl(\log(1-2x\alpha_k)+2x\alpha_k^T \Bigr) 
=\frac{1}{2\pi} \int_{\mathbb{R}} f_x(bg(y))\,dy,
\label{WEAKCVG1}
\end{equation}
where the spectral density $g$ is given by \eqref{SPECDENS} and for all $x\in \dR$ such that $|x|\leq 1/4A$,
$$ 
f_x(y)=\frac{1}{2}\Bigl(\log(1-2xy) + 2xy \Bigr).
$$
Hence, it follows from \eqref{WEAKCVG1} together with the elementary Taylor expansion of the logarithm
and classical complex analysis results that, for any integer $p \geq 2$,
\begin{equation*}
\lim_{T \rightarrow +\infty} \frac{1}{T} \sum_{k=1}^{\infty}(\alpha_k^T )^p
= \frac{1}{2\pi} \int_{\mathbb{R}} (bg(y))^p\,dy,
\end{equation*}
Therefore, we obtain the weak convergence \eqref{L-LIMNUT} on the class of functions $\cF$ from the Stone-Weierstrass theorem, which completes the proof
of Lemma \ref{L-LEMDEC}.
\demend

\vspace{-2ex}
\bibliographystyle{acm}

\end{document}